\documentclass[a4paper,11pt]{article}
\usepackage{ae} 
\usepackage[T1]{fontenc}
\usepackage[ansinew]{inputenc}
\usepackage[leqno,namelimits]{amsmath}
\usepackage{amsthm}
\usepackage{amssymb}
\usepackage{graphicx}

\newtheorem{proposition}{Proposition}[section]
\newtheorem{lemma}[proposition]{Lemma}
\newtheorem{corollary}[proposition]{Corollary}

\theoremstyle{definition}

\newtheorem{example}[proposition]{Example}

\theoremstyle{remark}
\newtheorem{remark}[proposition]{Remark}

\numberwithin{equation}{section}

 \DeclareMathOperator{\RE}{Re}

 \newcommand{\M}{\mathcal{M}}
 
 \newcommand{\K}{\mathcal{K}}
 
 \newcommand{\Real}{\mathbb{R}}
 \newcommand{\Complex}{\mathbb{C}}

 \newcommand{\abs}[1]{\left\vert#1\right\vert}
 \newcommand{\set}[1]{\left\{#1\right\}}

\author{ Aik. Aretaki and J. Maroulas\footnote{Department of Mathematics, National\, Technical \,University \,of
Athens, Zografou Campus, Athens 15780, Greece. E-mail address:
maroulas@math.ntua.gr.}}
\title{The higher rank numerical range \\of nonnegative matrices}
\begin{document}
\maketitle


\begin{abstract}
In this article the well known  ''Perron-Frobenius theory'' is investigated involving the higher rank numerical range $\Lambda_{k}(A)$ of an irreducible and entrywise nonnegative matrix $A$  and extending the notion of  elements of maximum modulus in $\Lambda_{k}(A)$. Further, an application of this theory to the $\Lambda_{k}(L(\lambda))$ of a Perron polynomial $L(\lambda)$ is elaborated via its companion matrix $C_{L}$.
\end{abstract}
{\small\textbf{Key words}:Perron-Frobenius theory, nonnegative matrix, Perron polynomial, higher rank numerical range, rank-k numerical radius.}

\section{Introduction}

Let $\M_{n}(\Complex)$ be the algebra of matrices $A=[a_{ij}]_{i,j=1}^{n}$
with entries $a_{ij}\in \Complex$ and $k\geq 1$ be a positive integer.
The \emph{k-rank numerical range} $\Lambda_{k}(A)$ of a matrix $A\in\M_{n}$ is defined  by
\begin{equation}\label{def1}
\Lambda_{k}(A)=\set{\lambda\in\Complex : PAP=\lambda P\,\, for\,\, some\,\, P\in\mathcal{P}_{k}},
\end{equation}
where $\mathcal{P}_{k}$ is the set of all  orthogonal projections $P$ of $\Complex^{n}$ onto any $k$-dimensional
subspace $\K$ of $\Complex^{n}$. Equivalently,
\begin{equation}\label{def2}
 \Lambda_{k}(A) = \set{\lambda\in\Complex : X^{*}AX=\lambda I_{k},\,\,X\in \M_{n,k}(\Complex),\,\,X^{*}X=I_{k}}.
\end{equation}
For any $k$, the sets $\Lambda_{k}(A)$ are generally called \emph{ higher rank numerical range}.
The concept of higher rank numerical range has been introduced by Choi \textit{et al} in \cite{Choi,C-H-K-Z,C-K-Z-q,C-K-Z}
and studied thoroughly by other researchers in \cite{Poon-Li-Sze,Li-Sze,Hugo}. Apparently, for $k=1$, $\Lambda_{k}(A)$ yields
the classical  \emph{numerical range} of a matrix $A$ \cite{Rao,H.J.T}, i.e.
\begin{equation}\label{rel4}
\Lambda_{1}(A)\equiv F(A)=\set{ x^{*}Ax : x\in \Complex^{n}, \,x^{*}x=1}
\end{equation}
and it is readily verified
\[
F(A)\supseteq\Lambda_{2}(A)\supseteq\ldots\supseteq\Lambda_{k}(A).
\]
Moreover, the notion of the \textit{numerical radius} for the numerical range,
\[
r(A)=\max{\set{\abs z: z\in F(A)}},
\]
has been extended to the \textit{rank k-numerical radius}
\[
r_{k}(A)=\max{\set{\abs z: z\in \Lambda_{k}(A)}},
\]
considering $r_{k}(A)=-\infty$, whenever $\Lambda_{k}(A)$ is an empty set \cite{Poon-Li-Sze}. Note that $r(A)\geq r_{k}(A)$ and $r(A)\geq\rho(A)$, where $\rho(\cdot)$ denotes the spectral radius of a matrix, i.e. $\rho(\cdot)=\max\set{\abs\lambda: \lambda\in\sigma(\cdot)}$, with $\sigma(\cdot)$ to be the spectrum of a matrix.

We mention that an $n\times n$ matrix $A$ is said to be \textit{nonnegative} when each $a_{ij}\geq0$, and this is denoted by writing $A\geq0$. Similarly, $A$ is said to be \textit{positive} whenever each $a_{ij}>0$, denoted by $A>0$. The matrix $A\in\M_{n}$ is called  \emph{reducible}  when there is a permutation matrix $P$ such that
$$P^{T}AP=\begin{bmatrix}
            R & S \\
            0 & T \\
          \end{bmatrix},
$$
where $R, T$ are both square. For $n=1$, should be $A=0$. Otherwise, $A$ is said to be \emph{irreducible}. If $A$ is nonnegative and irreducible having $q>1$ eigenvalues of maximum modulus, then it is called \textit{imprimitive} and $q$ is referred to as \textit{index of imprimitivity}. In case $q=1$, $A$ is characterized as \textit{primitive}.

It is  well known that Perron-Frobenius theory concerns the spe\-ctral properties of positive and nonnegative matrices, namely the existence of positive or nonnegative eigenvalues and eigenvectors \cite{H.J.}. In addition, Issos'  treatment contributes extensions of the Perron-Frobenius theorem to the numerical range of a nonnegative and irreducible matrix $A$, relating the $\rho(A)$ with the  $r(A)$ \cite{Issos}. These results give the motivation for further investigation of $\Lambda_{k}(A)$ in the case of a nonnegative and irreducible matrix $A$ (section 2).

In the next section, we present applications of the Perron-Frobenius theory derived for a matrix polynomial,  considering the higher rank numerical range of matrix polynomials.

\section{Nonnegative and irreducible matrices}

According to  Issos' main theorems \cite[Th.4,7]{Issos}, if the numerical range $F(A)$ of a nonnegative and irreducible matrix $A$ has $q$ maximal elements, then they are equally spaced around a circle centered at the origin through a constant angle, with one of them lying  on the positive real axis. In particular,
\[
\mathcal{F}(A)=\set{r(A)e^{\textbf{i}\frac{2\pi t}{q}}: t=0, \ldots, q-1}
\]
is exactly the set of all the maximal elements in $F(A)$ and to the numerical radius $r(A)\in F(A)$ there always corresponds a positive unit ve\-ctor. Furthermore, Issos also proved that the cardinality of the set $\mathcal{F}(A)$ coincides with the index of imprimitivity of $A$.

In this section, we investigate to what extend do Issos' results apply to  $\Lambda_{k}(A)$ of a nonnegative and irreducible matrix $A$, taking into consideration that it is a non empty set \cite{Poon-Li-Sze}. At this point, we define the \emph{maximal elements} in $\Lambda_{k}(A)$ to constitute the set:
\[
\mathcal{F}_{k}(A)=\set{z\in\Lambda_{k}(A) : \abs z=r_{k}(A)},
\]
which for $k=1$ is equal to $\mathcal{F}(A)$.
Although the numerical range $F(A)$ of a nonnegative matrix $A$ always contains the numerical radius $r(A)>0$ \cite[Th.1]{Issos}, it is not generally true that $r_{k}(A)\in\Lambda_{k}(A)$, as we may observe in the following example.\\
\begin{example}
Let the $8\times 8$ nonnegative irreducible  matrix
     $A=\left[\begin{smallmatrix}
     0  & 0   &  3  &   0  &   0  &   1  &   0  &   0\\
     5  & 0   &  0  &   0  &   0  &   0  &   1  &   0\\
     0  & 0   &  0  &   1  &   9  &   0  &   0  &   6\\
     0  & 1   &  0  &   0  &   0  &   0  &   0  &   0\\
     0  & 1   &  0  &   0  &   0  &   0  &   0  &   0\\
     0  & 0   &  0  &   4  &   1  &   0  &   0  &   2\\
     0  & 0   &  1  &   0  &   0  &   3  &   0  &   0\\
     0  & 5   &  0  &   0  &   0  &   0  &   0  &   0\\ \end{smallmatrix}\right]$.
The outer curve in the following figure illustrates the boundary of $F(A)$, whereas the second and third inner curves illustrate the boundary of $\Lambda_{2}(A)$ and $\Lambda_{3}(A)$, respectively.
Apparently, $0<r(A)\in F(A)$, but in the figure, we easily recognize that $0<r_{2}(A)\notin\Lambda_{2}(A)$ and $0<r_{3}(A)\notin\Lambda_{3}(A)$.
Note that $A$ has 4 maximal eigenvalues, which are marked by ''+'' and $F(A)$ has 4 maximal elements, as well.

\begin{center}\includegraphics[width=0.35\textwidth]{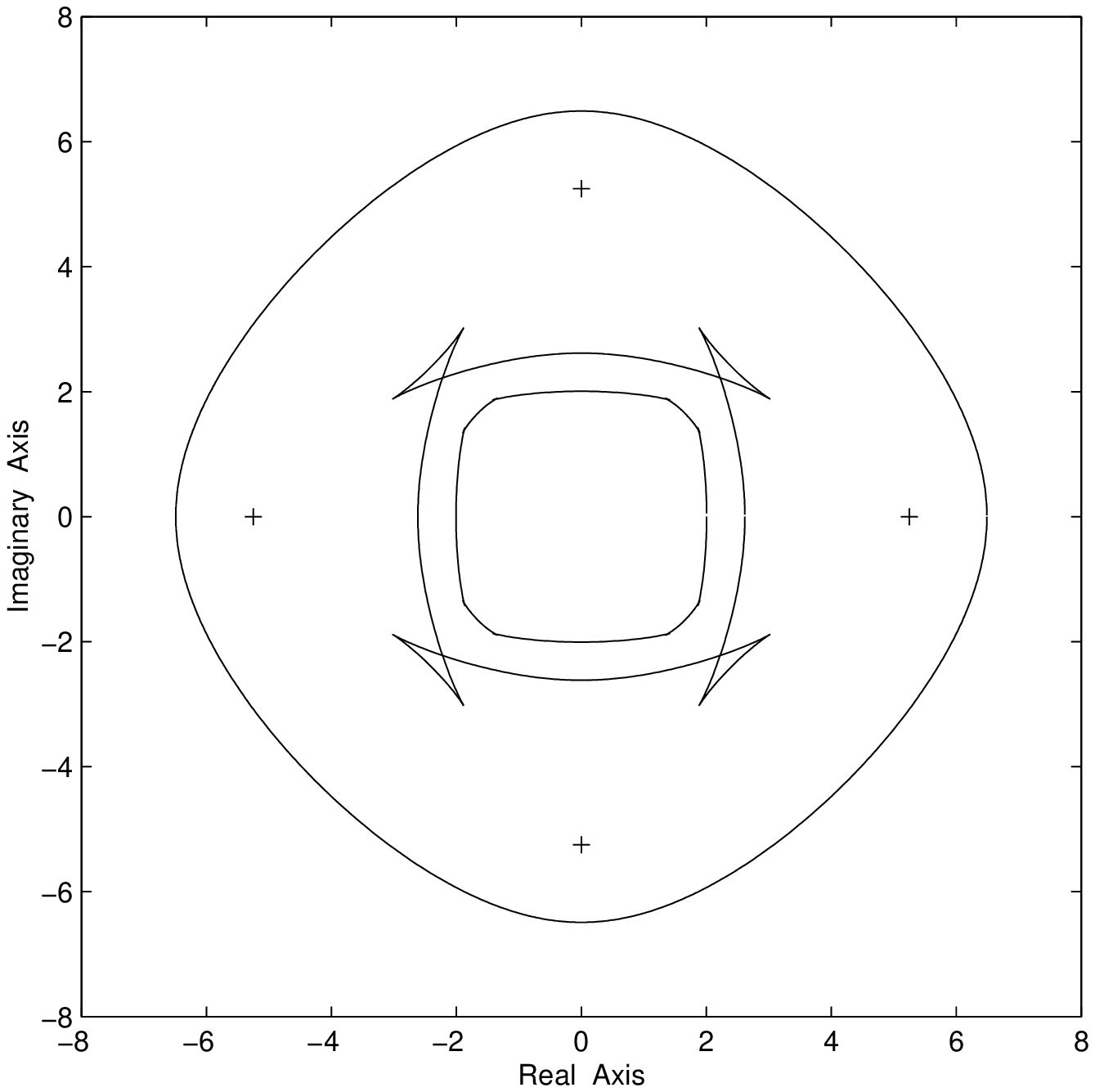}\end{center}                                      
\end{example}

The following Lemma generalizes a familiar condition for  rotational invariance and  symmetry about the origin to the case of the higher rank numerical range.

\begin{lemma}\label{lem1}
Let $A\in\M_{n}(\Complex)$ be permutation similar (hence unitarily similar) to the matrix
\begin{equation}\label{eq1}
C=\begin{bmatrix}
  0 & C_{12} & 0 & \cdot & 0 \\
  0 & 0 & C_{23}  & \cdot & 0 \\
  \cdot & \cdot & \cdot & \cdot & \cdot \\
  0 & \cdot & \cdot & 0 & C_{q-1,q} \\
  C_{q1} & 0 & \cdot & \cdot & 0 \\
\end{bmatrix}
\end{equation}
with the zero  blocks along the main diagonal be square.
For $j=1, \ldots, k$ such that $r_{j}(A)\neq\set{-\infty,0}$, we have:
\begin{description}
\item[ \,\,\textbf{I.}] $\Lambda_{j}(A)=\Lambda_{j}(e^{\textbf{i}\frac{2\pi t}{q}}A)$, for $t=0,1,\ldots, q-1$

\item[ \textbf{II.}] $\mu\in\Lambda_{j}(A)$ if and only if $\mu e^{\textbf{i}\frac{2\pi t}{q}}\in\Lambda_{j}(A)$, for $t=0, 1,\ldots, q-1$

\item[\textbf{III.}] $\Lambda_{j}(A)$ is symmetric with respect to the origin, if $q=2t$.
\end{description}
\end{lemma}
\begin{proof}
\textbf{I.} By the proof of Theorem 6 in \cite{Issos}, $D^{-1}CD=e^{\textbf{i}\theta}C$, where $D=I_{n_{1}}\oplus e^{\textbf{i}\theta}I_{n_{2}}\oplus \ldots\oplus e^{\textbf{i}(q-1)\theta}I_{n_{q}}$ is unitary diagonal with $\theta=\frac{2\pi t}{q}$ and $n_{1}+\ldots+n_{q}=n$, where $n_{l}$ $(l=1, \ldots, q)$ is the dimension of the $l$-th diagonal block of $C$. Hence, $e^{\textbf{i}\theta}A=(P^{T}DP)^{-1}A(P^{T}DP)$, where $P$ is the permutation matrix such that $A=P^{T}CP$, and  $\Lambda_{j}(A)=\Lambda_{j}(e^{\textbf{i}\theta}A)$, for any $j=1, \ldots, k$.\\
\textbf{II.} By statement (I), we have that
\[
\mu\in\Lambda_{j}(A)\Leftrightarrow e^{\textbf{i}\frac{2\pi t}{q}}\mu\in\Lambda_{j}(e^{\textbf{i}\frac{2\pi t}{q}}A)=\Lambda_{j}(A),
\]
for all $t=0, 1, \ldots, q-1$ and $j=1, \ldots, k$.\\
\textbf{III.} Let $q=2t$, then by (II), $\mu\in\Lambda_{j}(A)$ if and only if $e^{\textbf{i}\frac{2\pi t}{2t}}\mu\in\Lambda_{j}(A)$, i.e. $e^{\textbf{i}\pi}\mu=-\mu\in\Lambda_{j}(A)$ for $j=1, \ldots, k$.
\end{proof}

\begin{example}
Let the matrix $A=\left[\begin{smallmatrix}
  0 & 0 & i & 0 \\
  0 & 0 & -i & 0 \\
  0 & 0 & 0 & 2 \\
  3+2i & 1 & 0 & 0 \\
\end{smallmatrix}\right]$ as in \eqref{eq1} with $q=3$. The boundary of $\Lambda_{2}(A)$ is illustrated below by  the arched triangle. Apparently, $\Lambda_{2}(A)$ is rotationally invariant about the origin through an angle of $\frac{2\pi t}{3}$ for $t=0,1,2$, confirming  Lemma \ref{lem1}(I,II).
\begin{center}
    \includegraphics[width=0.35\textwidth]{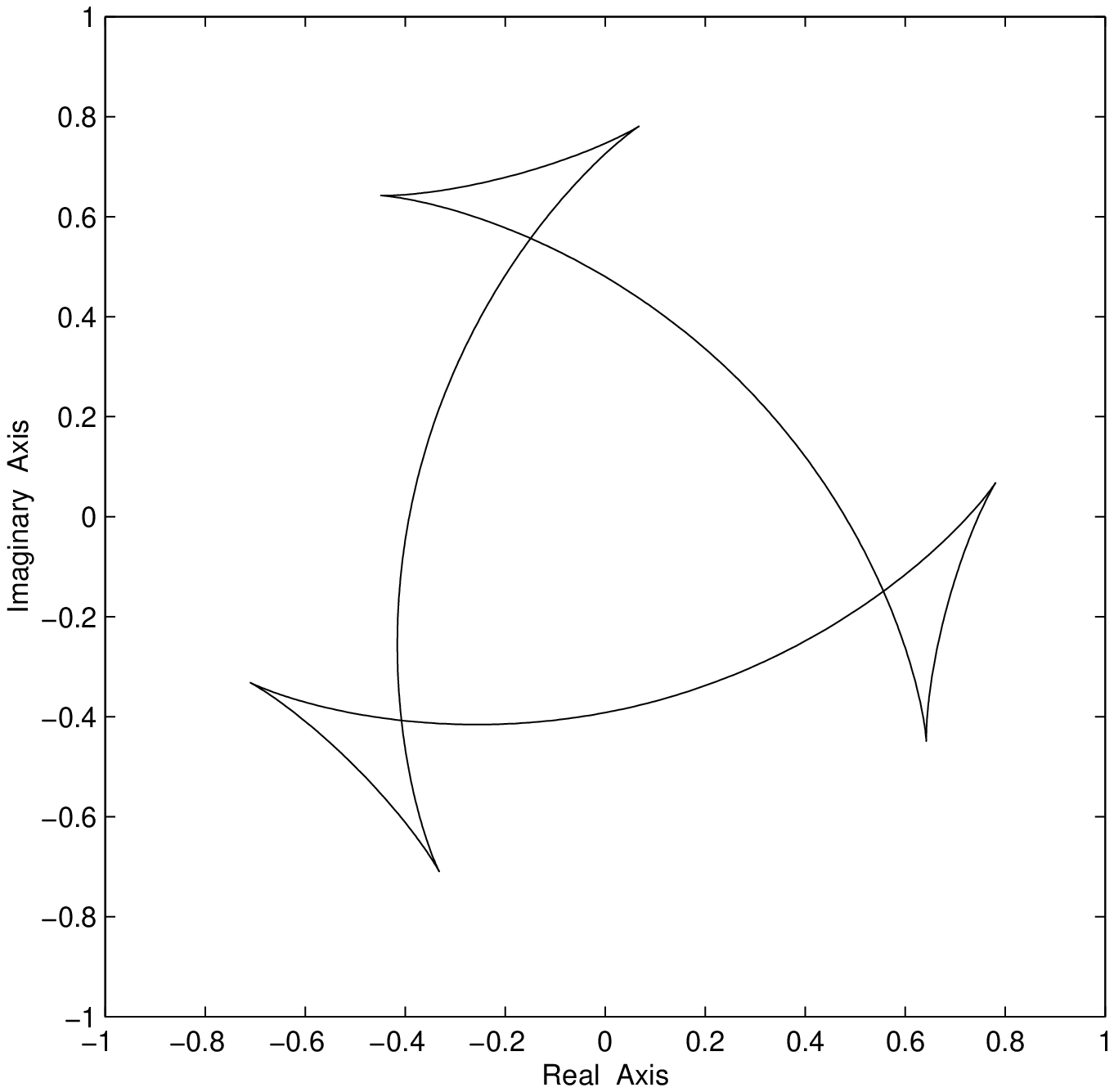}                                         
\end{center}
\end{example}
Any matrix $A$ being permutation similar to the matrix $C$ in \eqref{eq1} is said to be \emph{$q$-cyclic} and the largest positive integer $q$ such that $A$ is $q$-cyclic is called the \emph{cyclic index} of $A$.

Regarding  the number of  maximal elements in $\Lambda_{k}(A)$ and their location on the complex plane, we refer to  the following result.
\begin{proposition}\label{pr3}
Let $A\in\M_{n}(\Real)$ be imprimitive with index of imprimitivity $q>1$ such that $r_{k}(A)>0$. Then
\begin{equation}\label{eq2}
\mathcal{F}_{j}(A)=\set{r_{j}(A)e^{\textbf{i}\,(\theta_{j}+\frac{2\pi t}{q})}: t=0, \ldots, q-1},
\end{equation}
for every $j=1, \ldots, k$ with $\theta_{j}=0$ or $\theta_{j}=\frac{\pi}{q}$.
\end{proposition}
\begin{proof}
Since $q>1$ is the index of imprimitivity of $A$, there is a permutation matrix $P$ such that $$P^{T}AP=\begin{bmatrix}
  0 & C_{12} & 0 & \cdot & 0 \\
  0 & 0 & C_{23}  & \cdot & 0 \\
  \cdot & \cdot & \cdot & \cdot & \cdot \\
  0 & \cdot & \cdot & 0 & C_{q-1,q} \\
  C_{q1} & 0 & \cdot & \cdot & 0 \\
\end{bmatrix}.$$
Let $\theta_{j}\in[0,2\pi)$ be the principal argument such that  $0<r_{j}(A)\in\Lambda_{j}(e^{-\textbf{i}\theta_{j}}A)$ with $j=1, \ldots, k$. Then, by Lemma \ref{lem1}(II), we have $r_{j}(A)e^{\textbf{i}(\theta_{j}+\frac{2\pi t}{q})}\in\Lambda_{j}(A)$ for $t=0, 1, \ldots, q-1$ and $j=1, \ldots, k$, whereupon we obtain
\[
\mathcal{F}_{j}(A)\supseteq\set{r_{j}(A)e^{\textbf{i}(\theta_{j}+\frac{2\pi t}{q})}: t=0, \ldots, q-1}
\]
for every $j=1, \ldots, k$.\\
The index of imprimitivity $q$ is equal to the largest positive integer  such that $A$ is unitarily diagonally similar to the matrix
$e^{\textbf{i}\frac{2\pi}{q}}A$,  equivalently, the matrices $e^{-\textbf{i}\theta_{j}}A$ and $e^{\textbf{i}(-\theta_{j}+\frac{2\pi}{q})}A$
are unitarily diagonally similar for the largest positive integer q. Therefore, the set
\[
\set{0,\frac{2\pi}{q}, \ldots, \frac{2\pi(q-1)}{q}}
\]
is the cyclic group modulo $2\pi$ of the largest order, concluding that there does not exist $\phi=\frac{2\pi}{p}<\frac{2\pi}{q}$ such that $r_{j}(A)e^{\textbf{i}\phi}\in\Lambda_{j}(e^{-\textbf{i}\theta_{j}}A)$ for  $j=1, \ldots, k$. Hence we establish the equality
\begin{equation}\label{eq3}
\mathcal{F}_{j}(A)=\set{r_{j}(A)e^{\textbf{i}(\theta_{j}+\frac{2\pi t}{q})}: t=0, \ldots, q-1}.
\end{equation}
In addition, if we denote by $\overline{\,\,\,\,\cdot\,\,}$ the conjugate of a set,  it is clear that $\Lambda_{j}(A)=\overline{\Lambda_{j}(A)}$, since $A\in\M_{n}(\Real)$, i.e. $\Lambda_{j}(A)$ is symmetric with respect to the real axis.  Due to this symmetry and the equation \eqref{eq3}, if we consider $\theta_{j}\neq0$, we obtain $2\pi-\theta_{j}=\theta_{j}+\frac{2\pi(q-1)}{q}$. Hence,  $\theta_{j}=\frac{\pi}{q}$ and the proof is complete.
\end{proof}
In view of the preceding proposition, the number of elements of maximum modu\-lus in each $\Lambda_{j}(A)$ of a nonnegative and irreducible matrix $A$, for $j=1, 2, \ldots, k$, is equal to the index of imprimitivity $q$. Especially, they are all successively distributed around a circle centered at the origin through the constant  angle of $\frac{2\pi}{q}$. Furthermore, it is also observed that after a clockwise rotation of $\Lambda_{j}(A)$ about the origin through the angle of $\frac{\pi}{q}$, it is achieved $0<r_{j}(A)\in\Lambda_{j}(e^{-\textbf{i}\frac{\pi}{q}}A)$.

On the other hand, if we consider the primitive class of nonnegative matrices $A$ $(q=1)$, then the sets $\Lambda_{j}(A)$, for $j=2, \ldots, k$ do not necessarily have only one ma\-ximal element, contrary to the $F(A)$. This implication is justified by the symmetry of the sets $\Lambda_{j}(A)$ with respect to the real axis, when $r_{j}(A)e^{\textbf{i}\theta_{j}}\in\Lambda_{j}(A)$ for some $\theta_{j}\in(0,2\pi)$, then also $r_{j}(A)e^{-\textbf{i}\theta_{j}}\in\Lambda_{j}(A)$.\\
\begin{example}
Let the $8\times 8$ imprimitive  matrix
     $A=\left[\begin{smallmatrix}
     0 & 0 & 2 & 0 & 0 & 6 & 0 & 0\\
     1 & 0 & 0 & 0 & 0 & 0 & 7 & 0\\
     0 & 0 & 0 & 2 & 3 & 0 & 0 & 4\\
     0 & 3 & 0 & 0 & 0 & 0 & 0 & 0\\
     0 & 3 & 0 & 0 & 0 & 0 & 0 & 0\\
     0 & 0 & 0 & 4 & 0 & 0 & 0 & 2\\
     0 & 0 & 1 & 0 & 0 & 3 & 0 & 0\\
     0 & 9 & 0 & 0 & 0 & 0 & 0 & 0\\ \end{smallmatrix}\right]$
with index of imprimitivity $q=4$. The boundaries of $F(A)\equiv\Lambda_{1}(A)\supseteq\Lambda_{2}(A)\supseteq\Lambda_{3}(A)$ are illustrated below on the left figure and we verify that all $\Lambda_{j}(A)$  have $q=4$ maximal elements $(j=1, 2, 3)$, which are equally spaced through the angle of $t\frac{\pi}{2}$, $t=0, 1, 2, 3$. Specifically, $r_{2}(A)e^{\textbf{i}\frac{\pi}{4}}\in\Lambda_{2}(A)$, $r_{3}(A)\in\Lambda_{3}(A)$. \\
For the primitive case, let the nonnegative irreducible  matrix
$B=\left[\begin{smallmatrix}
     0 & 1 & 0 & 0 & 0 & 0\\
     0 & 0 & 1 & 0 & 0 & 0\\
     0 & 0 & 0 & 1 & 0 & 0\\
     0 & 0 & 0 & 0 & 1 & 0\\
     0 & 0 & 0 & 0 & 0 & 1\\
     1 & 1 & 1 & 0 & 0 & 0\\ \end{smallmatrix}\right]$
with only one maximal eigenvalue. The boundaries of $F(B)\supseteq\Lambda_{2}(B)$ are illustrated on the right figure, revealing that $\Lambda_{2}(B)$ has two symmetric maxi\-mal elements, with respect to the negative real semi axis, whereas $F(B)$ has only one maximal element $r(B)$.
The maximal eigenvalues of $A$ and $B$ are marked by ''+'' in both figures.
\begin{center}
  \begin{tabular}{cc}
    \includegraphics[width=0.35\textwidth]{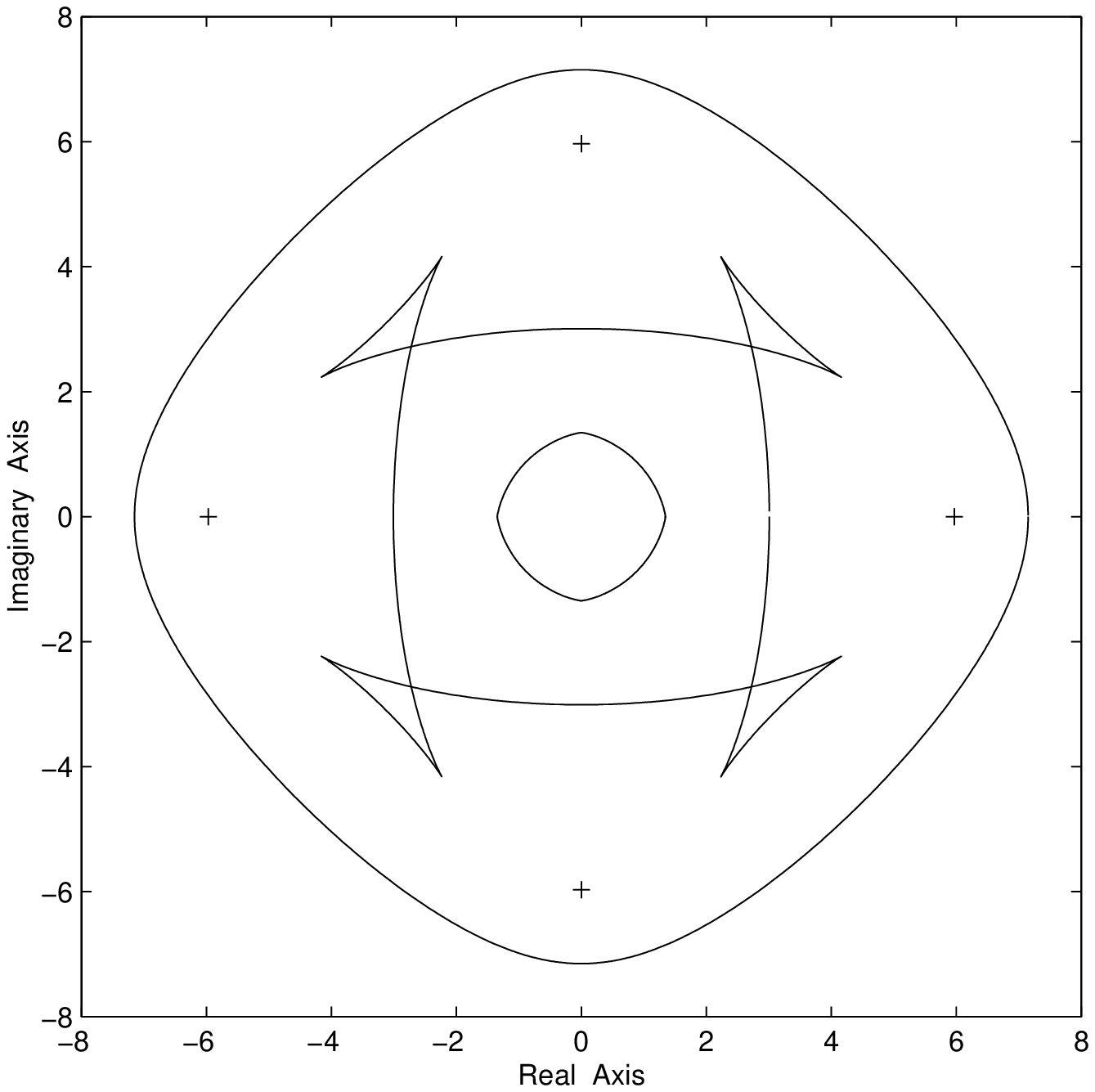}                                                     
    \includegraphics[width=0.35\textwidth]{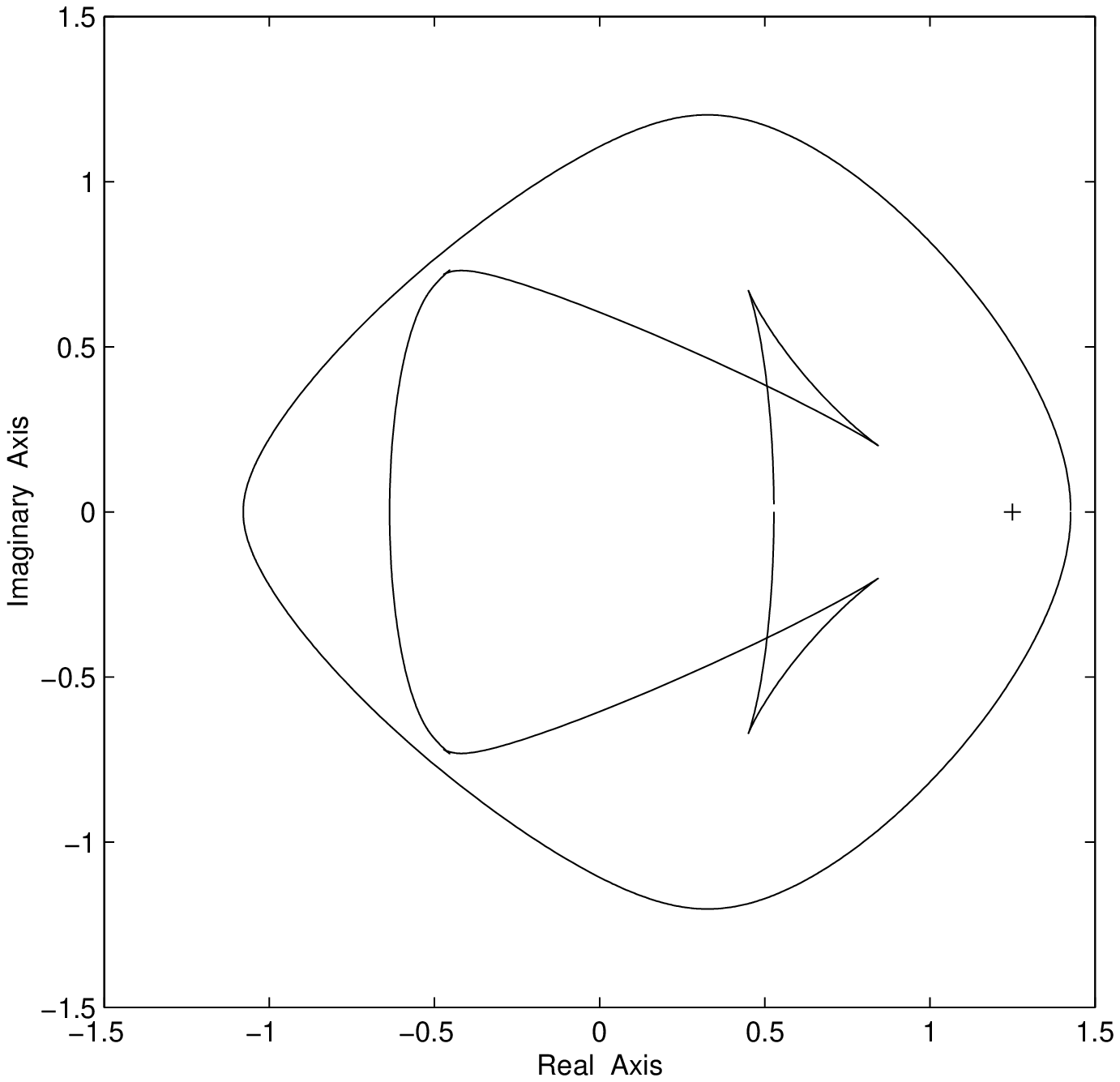}                                                    
    \end{tabular}
\end{center}
\end{example}

\begin{remark}
It is quite interesting to note that the $n$-cyclic permutation matrix $P_{n}=\left[ \begin{smallmatrix}0 & I_{n-1} \\1 & 0 \\ \end{smallmatrix}\right]\in\M_{n}(\Real)$ is a special form of imprimitive matrix of index $n$, with spectrum  the $n$-th roots of unity $z_{t}=e^{2\pi t\textbf{i}/n}$, $t=0, \ldots, n-1$. Apparently, $P_{n}$ is unitary and $F(P_{n})=co(\set{z_{0}, z_{1}, \ldots, z_{n-1}})$, where $co(\cdot)$ denotes the convex hull of a set. By corollary 2.8 in \cite{Gau}, whenever $2k<n$, we have that
\[
\Lambda_{k}(P_{n})=co(\set{\tilde{z}_{0}, \tilde{z}_{1}, \ldots, \tilde{z}_{n-1}}),
\]
with $\tilde{z}_{t}$ $(t=0, \ldots, n-1)$ to be intersection points of the line segments $[z_{t},z_{t+k}]$ and $[z_{t+1},z_{t+n-k+1}]$, when $z_{j}=z_{j-n}$ for $j>n-1$.\\
In addition, for the powers of $P_{n}$,  we have that $\Lambda_{k}(P_{n}^{\alpha})=\Lambda_{k}(P_{n})$ $(1\leq\alpha\leq n-1)$, since the matrices  $P_{n}^{\alpha}$ are permutation similar to $P_{n}$. Also, $P_{n}^{n}=I_{n}$, whereupon $\Lambda_{k}(P_{n}^{n})=\set{1}$.\\
For instance, $F(P_{5})$ and $\Lambda_{2}(P_{5})$ for the $5$-cyclic permutation $P_{5}=\left[ \begin{smallmatrix}0 & 1 & 0 & 0 & 0 \\
0 & 0 & 1 & 0 & 0 \\ 0 & 0 & 0 & 1 & 0 \\0 & 0 & 0 & 0 & 1 \\1 & 0 & 0 & 0 & 0 \\ \end{smallmatrix}\right]$ are illustrated below.
\begin{center}
    \includegraphics[width=0.35\textwidth]{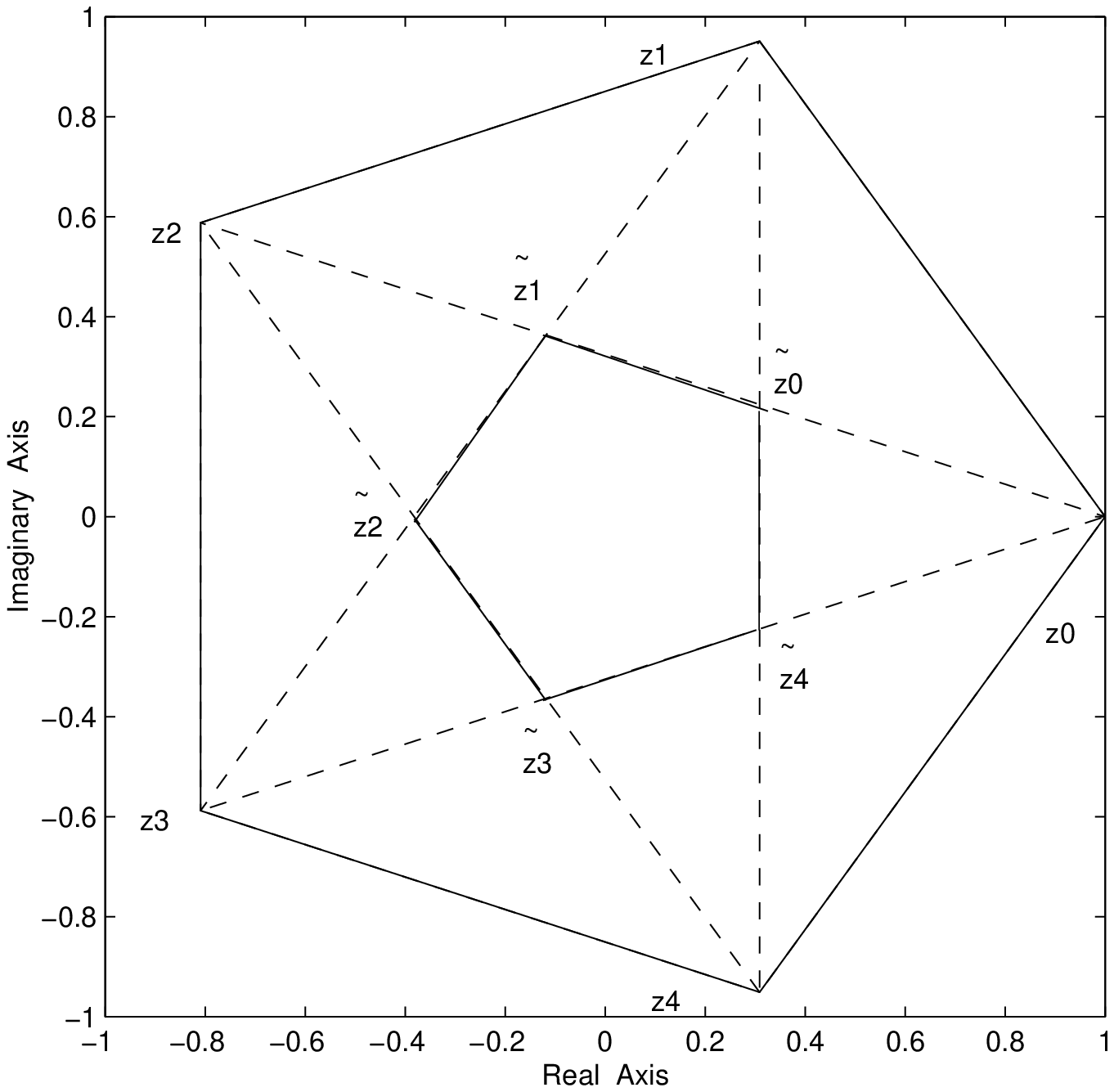}                                                           
\end{center}
\end{remark}

\begin{lemma}\label{lem2}
Let $A\in\M_{n}(\Real)$ be nonnegative with  irreducible hermitian part $H(A)$ such that $0<r_{j}(A)\in\Lambda_{j}(e^{-\textbf{\textrm{i}}\theta_{j}}A)$ for some $\theta_{j}\in[0,2\pi)$ and every $j=1, 2, \ldots, k$. If there exists an angle $\phi\in\Real$ such that $r(A)e^{\textbf{i}\phi}\in F(A)$, then $r_{j}(A)e^{\textbf{i}\phi}\in\Lambda_{j}(e^{-\textbf{i}\theta_{j}}A)$ for every $j=1, \ldots, k$.
\end{lemma}
\begin{proof}
Suppose $r(A)e^{\textbf{i}\phi}\in F(A)$ for some $\phi\in\Real$, then by proposition 3.7 in \cite{Li-Tam-Wu} we have that $e^{-\textbf{i}\phi}A=D^{-1}AD$ for some unitary diagonal matrix $D$. Obviously, $0<r_{j}(A)\in\Lambda_{j}(e^{-\textbf{i}\theta_{j}}A)=\Lambda_{j}(e^{-\textbf{i}(\theta_{j}+\phi)}A)$, therefore $r_{j}(A)e^{\textbf{i}\phi}\in\Lambda_{j}(e^{-\textbf{i}\theta_{j}}A)$ for every $j=1, \ldots, k$.
\end{proof}

Investigating $\Lambda_{k}(A)$  of a nonnegative matrix $A$ in terms of the irreducibi\-lity of the hermitian part $H(A)$ of the matrix $A$, we extend an analogous discussion developed in \cite{Li-Tam-Wu,M-Ps-Ts}.
\begin{proposition}\label{pr4}
Let $A\in\M_{n}(\Real)$, $A\geq0$ with $H(A)$ be irreducible and $r_{k}(A)>0$. Then either $\mathcal{F}_{j}(A)$ coincides with the circle $\mathcal{S}(0,r_{j}(A))$  for every $j=1, \ldots, k$ or $\mathcal{F}_{j}(A)=\set{r_{j}(A)e^{\textbf{i}\,(\theta_{j}+\frac{2\pi t}{q})}: t=0, \ldots, q-1}$ for every $j=1, \ldots, k$, where $\theta_{j}=0$ or $\theta_{j}=\frac{\pi}{q}$  and $q>1$ is the largest positive integer such that $A$ is diagonally  similar to the matrix $e^{\textbf{i}\frac{2\pi}{q}}A$.
\end{proposition}
\begin{proof}
Suppose that for each $j=1, \ldots, k$ we have $0<r_{j}(A)\in\Lambda_{j}(e^{-\textbf{i}\theta_{j}}A)$ for some principal argument $\theta_{j}\in[0,2\pi)$. Due to proposition 3.11 in \cite{Li-Tam-Wu} and corollary 3.6 in \cite{M-Ps-Ts}, either  $\mathcal{F}(A)=\mathcal{S}(0,r(A))$ or $\mathcal{F}(A)=\set{r(A)e^{\textbf{i}\frac{2\pi t}{q}}: \right.$ $\left. t=0, \ldots, q-1}$, where $q$ is the maximum positive integer such that $A$ is unitarily diagonally similar to the matrix $e^{\textbf{i}\frac{2\pi}{q}}A$. The first equality of the sets  indicates that $r(A)e^{\textbf{i}\phi}\in F(A)$ for every angle $\phi\in\Real$, thus by Lemma \ref{lem2}, $r_{j}(A)e^{\textbf{i}\phi}\in\Lambda_{j}(e^{-\textbf{i}\theta_{j}}A)$ for every angle $\phi\in\Real$, concluding $\mathcal{F}_{j}(A)=\mathcal{S}(0,r_{j}(A))$ for all $j=1, \ldots, k$. For the second case where $r(A)e^{\textbf{i}\frac{2\pi t}{q}}\in F(A)$,  Lemma \ref{lem2} verifies $r_{j}(A)e^{\textbf{i}\frac{2\pi t}{q}}\in\Lambda_{j}(e^{-\textbf{i}\theta_{j}}A)$ and then
\[
\set{r_{j}(A)e^{\textbf{i}(\theta_{j}+\frac{2\pi t}{q})}: t=0, \ldots, q-1}\subseteq\mathcal{F}_{j}(A).
\]
The equality of the sets, with $\theta_{j}=0$ or $\frac{\pi}{q}$ is established similarly as in the proof of Proposition \ref{pr3}, since $q$ is identified with the largest posi\-tive integer such that $e^{-\textbf{i}\theta_{j}}A$ is unitarily
diagonally similar to the matrix $e^{\textbf{i}(-\theta_{j}+\frac{2\pi}{q})}A$ and $\Lambda_{j}(A)$ are symmetric with respect to the real axis for every $j=1,\ldots, k$.
\end{proof}
\begin{corollary}\label{cor4a}
Let $A\in\M_{n}(\Real)$, $A\geq0$ with  irreducible hermitian part. If $0<r_{j}(A)\notin\Lambda_{j}(A)$ for some $j=2, \ldots, k$ and $\Lambda_{j}(A)$  is not a circular disc, then $\Lambda_{j}(A)$ is symmetric with respect to the lines $\mathcal{L}_{\pm}=\set{ze^{\pm\textbf{i}\frac{\pi}{q}}: z\in\Real}$, where $q$ is the largest positive integer such that $A$ is diagonally similar to $e^{\textbf{i}\frac{2\pi}{q}}A$.
\end{corollary}
\begin{proof}
By Proposition \ref{pr4},
$\mathcal{F}_{j}(A)=\set{r_{j}(A)e^{\textbf{i}\frac{\pi(2t+1)}{q}}: t=0, \ldots, q-1},$
whereupon
\[
\mathcal{F}_{j}(e^{\pm\textbf{i}\frac{\pi}{q}}A)=\set{r_{j}(A)e^{\textbf{i}\frac{2\pi t}{q}}: t=0, \ldots, q-1}\subseteq\partial\Lambda_{j}(e^{\pm\textbf{i}\frac{\pi}{q}}A),
\]
for some $j=2, \ldots, k$. Clearly, $\Lambda_{j}(e^{\pm\textbf{i}\frac{\pi}{q}}A)$ is symmetric with respect to $\Real$, implying $\Lambda_{j}(A)=e^{\mp\textbf{i}\frac{\pi}{q}}\Lambda_{j}(e^{\pm\textbf{i}\frac{\pi}{q}}A)$ to be symmetric with respect to the lines $\mathcal{L}_{\pm}$.
\end{proof}

\begin{corollary}\label{cor4}
Let $A\in\M_{n}(\Real)$, $A\geq0$ with $H(A)$ be irreducible such that $r_{k}(A)>0$. If $F(A)$ is a circular disc, then $\Lambda_{j}(A)$ is also a circular disc for every $j=2, \ldots, k$.
\end{corollary}

Any matrix of the form \eqref{eq1} with $C_{q1}=0$ and $q>1$ is called to be a \emph{ block-shift matrix}. It is already known that for real nonnegative matrices with irreducible hermitian part, which can be put into the block-shift form by means of a permutation,  we obtain the circularity of the numerical range \cite[Th.1]{Tam-Yang}. Because of Corollary \ref{cor4}, it is immediate that this type of matrices also cha\-racterize the circularity of the higher rank numerical range.

\begin{proposition}
Let $A\in\M_{n}(\Real)$, $A\geq0$ with $H(A)$ be irreducible. If $A$ is permutation similar to a block-shift matrix, then $\Lambda_{j}(A)$ is identified with the circular disc $\mathcal{D}(0,r_{j}(A))$ for $j=1, \ldots, k$.
\end{proposition}

We should note that the conclusion of the preceding proposition does not hold in general for any complex matrix $A$ being unitarily similar to a block-shift matrix. For the general case, we present the next result.
\begin{proposition}
Let $A\in\M_{n}(\Complex)$ be unitarily similar to  $\begin{bmatrix}0 & A_{1} \\ 0 & 0\end{bmatrix}$ with $A_{1}\in\M_{m,n-m}(\Complex)$ and $rankA_{1}=k$. Then
\[
\Lambda_{j}(A)=\mathcal{D}(0,\frac{\sigma_{j}(A_{1})}{2})\,\,for \,\,j=1, \ldots, k,
\]
where $\sigma_{j}(A_{1})$ denotes the $j$-th largest singular value of $A_{1}$ and then $r_{j}(A)=\frac{\sigma_{j}(A_{1})}{2}$, for $j=1, \ldots, k$.
\end{proposition}
\begin{proof}
It only suffices to prove that $\Lambda_{j}(\begin{bmatrix}0 & A_{1}\\ 0 & 0\end{bmatrix})=\mathcal{D}(0,\frac{\sigma_{j}(A_{1})}{2})$, for $j=1, \ldots, k$, since $\Lambda_{j}(A)$ is  invariant under unitary equivalence. Therefore, by \cite{Li-Sze}
\[
\Lambda_{j}(\begin{bmatrix}0 & A_{1}\\ 0 & 0\end{bmatrix})=\bigcap_{\theta\in[0, 2\pi)}{e^{-\textbf{i}\theta}\set{z\in\Complex : \RE z \leq \lambda_{j}(H(\begin{bmatrix}0 & e^{\textbf{i}\theta}A_{1}\\ 0 & 0\end{bmatrix}))}},
\]
where  $\lambda_{j}(H(A))$ denotes the $j$-th largest eigenvalue of the  hermitian part $H(A)$ of matrix $A$. Hence
\[
\Lambda_{j}(\begin{bmatrix}0 & A_{1}\\ 0 & 0\end{bmatrix})=\bigcap_{\theta\in[0, 2\pi)}{e^{-\textbf{i}\theta}\set{z\in\Complex : \RE z \leq \frac{1}{2}\lambda_{j}(\begin{bmatrix}0 & e^{\textbf{i}\theta}A_{1}\\ e^{-\textbf{i}\theta}A_{1}^{*} & 0\end{bmatrix})}}.
\]
It is known that the eigenvalues of the hermitian matrix $\begin{bmatrix}0 & e^{\textbf{i}\theta}A_{1}\\ e^{-\textbf{i}\theta}A_{1}^{*} & 0\end{bmatrix}$ are the singular values $\sigma_{1}(A_{1})\geq\ldots\geq\sigma_{k}(A_{1})>0\geq\ldots\geq0>-\sigma_{k}(A_{1})\geq\ldots\geq
-\sigma_{1}(A_{1})$ with $k=rankA_{1}$ \cite{H.J.}, thus for $j=1, \ldots, k$
\[
\Lambda_{j}(\begin{bmatrix}0 & A_{1}\\ 0 & 0\end{bmatrix})=\bigcap_{\theta\in[0, 2\pi)}{e^{-\textbf{i}\theta}\set{z\in\Complex : \RE z \leq \frac{\sigma_{j}(A_{1})}{2}}}=\mathcal{D}(0,\frac{\sigma_{j}(A_{1})}{2}).
\]
\end{proof}

\section{Application to matrix polynomials}
An extension of the aforementioned results to the  higher rank numerical range of a matrix polynomial arises naturally and it is the purpose of this section. For this reason, we refer to \emph{Perron polynomials} $L(\lambda)$, which are $n\times n$ monic matrix polynomials of $m$th degree
\begin{equation}\label{per}
L(\lambda)=I\lambda^{m}-A_{m-1}\lambda^{m-1}-\ldots-A_{1}\lambda-A_{0},
\end{equation}
with $A_{j}$ to be nonnegative matrices, for $j=0, \ldots, m-1$. The higher rank numerical range $\Lambda_{k}(L(\lambda))$  of $L(\lambda)$
has been recently defined  in \cite{Aretaki} by the set
\begin{equation}\label{re11}
\Lambda_{k}(L(\lambda))=\set{\lambda\in\Complex: Q^{*}L(\lambda)Q=0_{k}\,\, \mathrm{for}\, \mathrm{some}\, Q\in\M_{n,k},\,Q^{*}Q=I_{k}},
\end{equation}
which for $k=1$ yields the numerical range
\begin{equation}\label{rel2}
w(L(\lambda))=\set{\lambda\in\Complex : x^{*}L(\lambda)x=0\,\, \mathrm{for}\,\, \mathrm{some}\,\, x\in\Complex^{n}, x^{*}x=1}.
\end{equation}
The notion of the Perron polynomial $L(\lambda)$ in \eqref{per} is equivalent to the nonnegativity of its $mn\times mn$ companion matrix
\[
C_{L}=\begin{bmatrix}
                          0 & I_{n} & 0 & \cdots & 0 \\
                          0 & 0 & I_{n} & \cdots & 0 \\
                          \vdots & & \ddots & \ddots & \vdots \\
                          0 & &  & & I_{n}\\
                          A_{0} & & \cdots & & A_{m-1} \\
                          \end{bmatrix},
\]
hence an extension of the main Perron-Frobenius theorem concerning the spectrum $\sigma(L)$ of $L(\lambda)$ and Issos' results on the numerical range $w(L(\lambda))$ are established via the companion matrix $C_{L}$ \cite{Ps-Ts}. Further, we denote the rank $k$-numerical radius of $\Lambda_{k}(L(\lambda))$
\[
r_{k}(L)=\max\set{\abs\lambda: \lambda\in\Lambda_{k}(L(\lambda))},
\]
with $r_{k}(L)=-\infty$ when $\Lambda_{k}(L(\lambda))=\emptyset$ and consequently, the set of maximal elements in $\Lambda_{k}(L(\lambda))$:
\[
\mathcal{F}_{k}(L)=\set{\lambda\in\Lambda_{k}(L(\lambda)): \abs\lambda=r_{k}(L)}.
\]
As we have noticed in section 2, clearly  $\Lambda_{k}(L(\lambda))$ does not always contain the element
$r_{k}(L)>0$.

The following lemma is a  generalization of a corresponding result in \cite{Ps-Ts}, which identifies $\Lambda_{k}(L(\lambda))$ with a specific
subset of $\Lambda_{k}(C_{L})$ of the companion matrix $C_{L}$ of a matrix polynomial $L(\lambda)$. We should recall that $\Lambda_{k}(C_{L})$ always
contains $\Lambda_{k}(L(\lambda))$ as it has been proved in \cite[Prop.16]{Aretaki}.

\begin{lemma}\label{lem3}
Let the $mn\times k$ matrix $Y(\lambda,Q)=(1,\lambda,\ldots,\lambda^{m-1})^{T}\otimes Q$ with $\lambda\in\Complex$ and $Q\in\M_{n,k}$.
Then
\[
\Lambda_{k}(L(\lambda))\cup\set{0} = \set{\lambda: Y^{*}(\lambda,Q)C_{L}Y(\lambda,Q)=\lambda I_{k},\,\, Q\in\M_{n,k}\,,\,Q^{*}Q=\frac{1}{c(\abs{\lambda}^{2})}I_{k}},
\]
where $c(\abs{\lambda}^{2})=(1+\abs{\lambda}^{2}+\abs{\lambda}^{4}+\ldots+\abs{\lambda}^{2(m-1)})^{1/2}$.
\end{lemma}
\begin{proof}
It is readily verified that
\[
Y^{*}(\lambda,Q)(C_{L}-\lambda I_{mn})Y(\lambda,Q) =\bar{\lambda}^{m-1}Q^{*}L(\lambda)Q
\]
and $Y^{*}(\lambda,Q)Y(\lambda,Q)=c(\abs{\lambda}^{2})Q^{*}Q$. Hence,
$\lambda_{0}\in\Lambda_{k}(L(\lambda))$ if and only if \\
$\sqrt{c(\abs{\lambda_{0}}^{2})}Q^{*}L(\lambda_{0})Q\sqrt{c(\abs{\lambda_{0}}^{2})}=0_{k}$
and equivalently when $Q^{*}L(\lambda_{0})Q=0_{k}$.
\end{proof}
\begin{proposition}
Let $L(\lambda)$ be a Perron polynomial as in \eqref{per} such that $r_{k}(L)>0$, with irreducible companion matrix $C_{L}$ and $q>1$ eigenvalues  of maximum modulus. Then the maximal elements in $\Lambda_{j}(L(\lambda))$ are exactly the set
\begin{equation}
\mathcal{F}_{j}(L)=\set{r_{j}(L)e^{\textbf{i}(\theta_{j}+\frac{2\pi t}{q})}, \,t=0, \ldots, q-1},
\end{equation}
where $\theta_{j}=0$ or $\theta_{j}=\frac{\pi}{q}$, for $j=1, \ldots, k$.
\end{proposition}
\begin{proof}
The companion matrix $C_{L}\geq0$ is an imprimitive matrix with index of imprimitivity $q>1$, since $\sigma(L)=\sigma(C_{L})$. Hence $q$ is the largest positive integer such that $C_{L}$ is unitarily diagonal to $e^{\textbf{i}\frac{2\pi}{q}}C_{L}$. Therefore,  $e^{\textbf{i}\frac{2\pi t}{q}}C_{L}=D^{-1}C_{L}D$ $(t=0, 1, \ldots, q-1)$, for some unitary diagonal $mn\times mn$ matrix $D=diag(D_{0},e^{\textbf{i}\frac{2\pi t}{q}}D_{0}, \ldots,$ $e^{\textbf{i}\frac{2\pi t(m-1)}{q}}D_{0})$, with $D_{0}\in\M_{n}$ to be unitary diagonal. Suppose $\theta_{j}\in[0,2\pi)$ is the principal argument such that $0<r_{j}(L)\in\Lambda_{j}(e^{-\textbf{i}\theta_{j}}L(\lambda))$, then by Lemma \ref{lem3} we have
\begin{eqnarray*}
0_{k} & = & Y^{*}(r_{j}(L)e^{\textbf{i}\theta_{j}},Q)(C_{L}-r_{j}(L)e^{\textbf{i}\theta_{j}} I_{mn})Y(r_{j}(L)e^{\textbf{i}\theta_{j}},Q)\\
& = & Y^{*}(r_{j}(L)e^{\textbf{i}\theta_{j}},Q)D^{-1}(e^{-\textbf{i}\frac{2\pi t}{q}}C_{L}-r_{j}(L)e^{\textbf{i}\theta_{j}}I_{mn})
DY(r_{j}(L)e^{\textbf{i}\theta_{j}},Q)\\
& = & Y^{*}(r_{j}(L)e^{\textbf{i}(\theta_{j}+\frac{2\pi t}{q})},D_{0}Q)(C_{L}-r_{j}(L)e^{\textbf{i}(\theta_{j}+\frac{2\pi t}{q})} I_{mn})Y(r_{j}(L)e^{\textbf{i}(\theta_{j}+\frac{2\pi t}{q})},D_{0}Q),
\end{eqnarray*}
with $Y^{*}(e^{\textbf{i}\frac{2\pi t}{q}}\lambda,D_{0}Q)Y(e^{\textbf{i}\frac{2\pi t}{q}}\lambda,D_{0}Q)=Y^{*}(\lambda,Q)Y(\lambda,Q)=I_{mn,k}$
for any scalar $\lambda$. Apparently,
\begin{equation}
\mathcal{F}_{j}(L)\supseteq\set{r_{j}(L)e^{\textbf{i}(\theta_{j}+\frac{2\pi t}{q})}, \,t=0, \ldots, q-1}.
\end{equation}
Further, as in the proof of Proposition \ref{pr3}, we establish the equality of the sets with $\theta_{j}=0$ or $\theta_{j}=\frac{\pi}{q}$ for $j=1, \ldots, k$, since the coefficients  of $L(\lambda)$ are real, whereupon $\Lambda_{k}(L(\lambda))$ is symmetric with respect to the real axis.
\end{proof}

\bibliographystyle{amsplain}

\begin{thebibliography}{99}

\bibitem{Aretaki} Aik. Aretaki and J. Maroulas, The higher rank numerical range of matrix polynomials,
10th Workshop on the ''Numerical Range and Numerical Radii'', Krakow, Poland, 2010, \emph{submitted for publication}.

\bibitem{Choi} M.D. Choi, M. Giesinger, J.A. Holbrook and D.W. Kribs, Geometry of
higher-rank numerical ranges, \emph{Linear and Multilinear Algebra},
\textbf{56}(1), 53-64, 2008.

\bibitem{C-H-K-Z} M.D. Choi, J.A. Holbrook, D.W. Kribs and K. Zyczkowski, Higher-rank
numerical ranges of unitary and normal matrices, preprint, http://arxiv.org/quant-ph/0608244.

\bibitem{C-K-Z-q} M.D. Choi, D.W. Kribs and K. Zyczkowski, Quantum error
correcting codes from the compression formalism, \textit{Reports on
Mathematical Physics}, \textbf{58}, 77-86, 2006.

\bibitem{C-K-Z} M.D. Choi, D.W. Kribs and K. Zyczkowski, Higher-rank
numerical ranges and compression problems, \emph{Linear Algebra
and its Applications}, \textbf{418}, 828-839, 2006.

\bibitem{Gau} H.L. Gau, C.K. Li, Y.T. Poon  and N.S. Sze, Quantum error correction and higher rank numerical ranges of normal matrices,
preprint, http://arxiv.org/0902.4869v1 [math.FA], 2009, SIAM J. Matrix Analysis and Applications, (to appear).

\bibitem{Rao} K.E. Gustafson and D.K.M. Rao, \textit{Numerical Range. The
Field of Values of Linear Operators and Matrices}, Springer-Verlag,
New York, 1997.

\bibitem{H.J.} R.A. Horn and C.R. Johnson, \textit{Matrix Analysis},
Cambridge University Press, Cambridge, 1985.

\bibitem{H.J.T} R.A. Horn and C.R. Johnson, \textit{Topics in Matrix Analysis},
Cambridge University Press, Cambridge, 1991.

\bibitem{Issos} J.N. Issos, \textit{The field of values of non-negative irreducible matrices}, Ph.D. Thesis, Auburn University, 1966.

\bibitem{Poon-Li-Sze} C.K. Li, Y.T. Poon and N.S. Sze, Condition for
the higher rank nume\-rical range to be non-empty, \textit{Linear and Multilinear Algebra}, \textbf{57}(4), 365-368,
2009.

\bibitem{Li-Sze} C.K. Li and N.S. Sze, Canonical forms, higher rank numerical
ranges, totally isotropic subspaces, and matrix equations, \textit{Proceedings of the American
Mathematical Society}, \textbf{136}, 3013-3023, 2008.

\bibitem{Li-Tam-Wu} C.K. Li, B.S. Tam and P.Y. Wu, The numerical range of a nonnegative matrix,
\emph{Linear Algebra and its Applications}, \textbf{350}, 1-23, 2002.

\bibitem{M-Ps-Ts} J. Maroulas, P.J. Psarrakos and M.J Tsatsomeros, Perron-Frobenius type results on the numerical range, \emph{Linear Algebra
and its Applications}, \textbf{348}, 49-62, 2002.

\bibitem{Ps-Ts} P.J. Psarrakos and M.J Tsatsomeros, A primer of Perron-Frobenius theory for matrix polynomials, \emph{Linear Algebra
and its Applications}, \textbf{393}, 333-351, 2004.

\bibitem{Tam-Yang} B.S. Tam and S. Yang. On matrices whose numerical ranges have circular or weak circular symmetry, \emph{Linear Algebra
and its Applications}, \textbf{302-303}, 193-221, 1999.

\bibitem{Hugo} H.J. Woerdeman, The higher rank numerical range is
convex, \textit{Linear and Multilinear Algebra}, \textbf{56}(1), 65-67, 2007.

\end{thebibliography}

\end{document}